\documentclass[11pt]{amsart}

\usepackage{amsmath,amsthm, amscd, amssymb, amsfonts}
\usepackage[usenames, dvipsnames]{xcolor}

\newcommand{\eq}{\normalcolor{}}

\newcommand{\Hc}{{\mathcal H}}

\newcommand{\fpd}{\mbox{\rm FPdim\,}}

\newcommand{\brp}{\mbox{\rm BrPic\,}}

\newcommand{\otk}{{\otimes_{\ku}}}

\newcommand{\nic}{{\mathfrak B}}
\newcommand{\qs}{{\mathfrak R}}
\newcommand{\os}{{\mathfrak O}}
\newcommand{\cab}{\underline{{\mathcal C}}}

\newcommand{\Ss}{{\mathcal S}}
\newcommand{\ot}{{\otimes}}
\newcommand{\inbi}{\operatorname{InnbiGal}}
\newcommand{\outb}{\operatorname{OutbiGal}}

\newcommand{\biga}{\operatorname{BiGal}}
\newcommand{\kc}{{\mathcal K}}
\newcommand{\Ac}{{\mathcal A}}

\newcommand{\ele}{{\mathcal L}}

\newcommand{\ca}{{\mathcal C}}

\newcommand\Comod{\operatorname{Comod}}

\newcommand{\Do}{{\mathcal D}}

\newcommand{\Bc}{{\mathcal B}}

\newcommand{\Fc}{{\mathcal F}}

\newcommand{\cop}{\rm{cop}}

\newcommand{\ku}{{\Bbbk}}

\newcommand{\Z}{{\mathbb Z}}
\newcommand{\Na}{{\mathbb N}}

\newcommand{\uno}{ \mathbf{1}}
\newcommand{\C}{{\mathbb C}}

\newcommand{\id}{\mbox{\rm id\,}}
\newcommand{\Id}{\mbox{\rm Id\,}}

\newcommand\Rep{\operatorname{Rep}}

\newcommand\co{\operatorname{co}}

\newcommand\Hom{\operatorname{Hom}}

\newcommand{\gr}{\mbox{\rm gr\,}}

\renewcommand{\_}[1]{\mbox{$_{\left( #1 \right)}$}}

\theoremstyle{plain}

\numberwithin{equation}{section}

\newtheorem{teo}{Theorem}[section]

\newtheorem{lema}[teo]{Lemma}

\newtheorem{cor}[teo]{Corollary}

\newtheorem{prop}[teo]{Proposition}

\theoremstyle{definition}

\newtheorem{defi}[teo]{Definition}

\theoremstyle{remark}

\newtheorem{rmk}[teo]{Remark}

\def\pf{\begin{proof}}

\def\epf{\end{proof}}

\theoremstyle{remark}

\begin{document}

\title[Crossed extensions of  supergroup algebras]
{Crossed extensions of the corepresentation category of finite supergroup algebras}
\author[ Mej\'ia Casta\~no and Mombelli  ]{Adriana Mej\'ia Casta\~no and Mart\'in Mombelli
 }
\address{
}
 \email{ sighana25@gmail.com}
\email{martin10090@gmail.com, mombelli@mate.uncor.edu
\newline \indent\emph{URL:}\/ http://www.mate.uncor.edu/$\sim$mombelli/welcome.html}

\begin{abstract}  We present explicit examples finite tensor categories that are $C_2$-graded extensions of the
corepresentation category of certain finite-dimensional non-semisimple Hopf algebras.

\bigbreak
\bigbreak
\bigbreak
\bigbreak
{\em Mathematics Subject Classification (2010): 18D10, 16W30, 19D23.}

{\em Keywords:  tensor category, Hopf algebra, BiGalois extensions}
\end{abstract}

\date{\today}
\maketitle

\section{Introduction}

Throughout this paper we shall work over an algebraically closed field $\ku$ of characteristic zero.

Given a finite group $\Gamma$ a (faithful) $\Gamma$-grading on a  finite tensor category $\Do$ is a decomposition 
$\Do=\oplus_{g\in \Gamma} \Do_g$, where $\Do_g$ are full Abelian subcategories of $\Do$ such that
\begin{itemize}
\item[$\bullet$]  $\Do_g\neq 0$;
 \item[$\bullet$] $\ot:\Do_g\times \Do_h\to \Do_{gh}$ for all
$g, h\in \Gamma.$ 
\end{itemize}
 In this case $\ca=\Do_e$ is a tensor subcategory of $\Do$. The tensor category $\Do$ is a $\Gamma$-\emph{extension} of $\ca$.   The category 
 $\Do_g$ is an invertible $\ca$-bimodule category  for any $g\in \Gamma$. This gives rise to a group 
 homomorphism $c:\Gamma\to \brp(\ca)$, where $\brp(\ca)$ is the so-called \emph{Brauer-Picard group} 
 of $\ca$ introduced in \cite{ENO3}. The Brauer-Picard group of a finite tensor category $\ca$ is the group of 
 equivalence classes of invertible exact $\ca$-bimodule categories.

 Given a finite group
 $\Gamma$ and  a fusion category $\ca$, $\Gamma$-extensions of $\ca$ were classified in \cite{ENO3}. Any such extension 
 depends on a group map $c:\Gamma\to \brp(\ca)$ and certain cohomological data. The problem of giving concrete
 examples of 
 $\Gamma$-extensions of a given finite tensor category $\ca$ is that, besides the cohomological obstructions, the explicit computation of the 
 Brauer-Picard group is needed. The computation of Brauer-Picard group is in general complicated.
 Some computations of this group were done in \cite{M2}, \cite{NR}, \cite{GS}.
 
 \medbreak
 A different version of $\Gamma$-extensions was studied in \cite{Ga1}. In \emph{loc. cit}. the author studies and classifies 
 $\Gamma$-gradings $\Do=\oplus_{g\in \Gamma} \Do_g$ such that there are equivalences $\Do_g\simeq \Do_e$ as 
 $\Do_e$-module 
 categories for any $g\in \Gamma$.
 Such extensions are called $\Gamma$-\emph{crossed products} and they are
 classified by equivalence classes of \emph{crossed systems} 
 of $\Gamma$ over $\ca$. A crossed system of $\Gamma$ over $\ca$ consists of a collection
$\Sigma=((a_*,\xi^a),(U_{a,b},\sigma^{a,b}),\gamma_{abc})_{a,b,c\in \Gamma}$ where
\begin{itemize}

\item[$\bullet$] $(a_*,\xi^a):\ca\to\ca$ are monoidal autoequivalences, with monoidal structure
 $$\xi^a_{X,Y}:a_*(X\ot Y)\to a_*(X)\ot a_*(Y),\quad X,Y\in\ca;$$
\item[$\bullet$]  invertible objects $U_{a,b}\in\C$;
\item[$\bullet$]  natural isomorphisms
$$\sigma^{a,b}_X:a_*b_*(X)\ot U_{a,b}\to U_{a,b}\ot (ab)_*X, \quad X\in\ca;$$
\item[$\bullet$]  isomorphisms $\gamma_{a,b,c}:a_*(U_{b,c})\ot U_{a,bc}\to U_{a,b}\ot U_{ab,c}$,
\end{itemize}
such that they satisfy certain conditions.  If $ \Sigma$ is a crossed system of $\Gamma$ over $\ca$ we define a new 
category $\ca(\Sigma)=\oplus_{a\in \Gamma}\ca_a$ as Abelian categories and
$\ca_a=\ca$ for all $a\in \Gamma$. Denote by $[V,a]$ the object $V\in\ca_a$. In \cite{Ga1}  the author introduces a new
tensor product on the category $\ca(\Sigma)$ given by
$$[V,a]\ot [W,b]=[V\ot a_*(W)\ot U_{a,b},ab],$$
for any $V, W\in \ca$, $a, b\in \Gamma$. The conditions of crossed system ensures that 
$\ca(\Sigma)$ is indeed a monoidal category.

\medbreak

The present paper is devoted to give explicit examples of $C_2$-crossed products, where $C_2$ is the cyclic group of two 
elements, of the category $\Comod(H)$ of
finite-dimensional $H$-comodules, where $H$ is a supergroup algebra. Part of the information needed to compute 
crossed systems
in this particular case is the computation of
tensor autoequivalences $F:\Comod(H)\to \Comod(H)$, thus we need to compute the group $\biga(H)$
of equivalence classes of biGalois objects over $H$ \cite{S2}. The group $\biga(H)$  is interesting from the 
Hopf algebraic point of view.  It was computed only for few examples, see  \cite{Bic}, \cite{Ca}, \cite{S3}. In this work we 
present a  technique  to compute the biGalois group for supergroup algebras. This 
technique is different from the one presented by Schauenburg in  \cite{S3}.

\medbreak

The examples of $C_2$-crossed products presented here are representation categories of quasi-Hopf algebras. 
We do not know how to compute those quasi-Hopf algebras explicitly. We believe that these tensor categories 
are not equivalent to the representation categories of a (usual) Hopf algebra. We will address this question in a 
forthcoming paper. 

\medbreak

The paper is organized as follows. In Section \ref{s-int} we give  the required notations. In Section 
\ref{spa} we describe the Hopf algebra structure  of the supergroup algebras introduced in \cite{AEG}. 
For any supergroup algebra we describe the projective covers of its simple objects. This description will be useful when computing 
certain Frobenius-Perron dimensions. In Section \ref{spa-hb} we  classify
biGalois objects for supergroup algebras. BiGalois objects are a fundamental piece of information to compute 
examples of crossed systems.  In Section \ref{section:crossed} we recall the definition of crossed product tensor 
category as introduced in \cite{Ga1} and how they are constructed from crossed systems. We also give a 
more concrete description of crossed systems in the case the tensor category is the category of corepresentations
of a finite-dimensional Hopf algebra. In Section \ref{section:crossed-spg} we give explicit examples of 
crossed systems of $C_2$ over a supergroup algebra and we describe the monoidal structure. We obtain 
eight non-equivalent tensor categories and we compute their Frobenius-Perron dimensions.

\section{Preliminaries and notation}\label{s-int}

 If $\Gamma$ is a finite
group and $\psi\in Z^2(\Gamma,\ku^{\times})$ is a 2-cocycle, there is another 2-cocycle $\psi'$
in the same cohomology class as $\psi$ such that
\begin{equation}\label{2-cocycl}
 \psi'(g,1)=\psi'(1,g)=1, \quad \psi'(g,g^{-1})=1, \quad \psi'(g,h)^{-1}=\psi'(h^{-1}, g^{-1}),
\end{equation}
for all $g,h\in \Gamma$. From now on, all elements in $Z^2(\Gamma,\ku^{\times})$ representing some
class in $H^2(\Gamma,\ku^{\times})$ will satisfy equation \eqref{2-cocycl}. For references in group cohomology 
see  \cite{Br}.

If $H$ is a Hopf algebra and $g\in G(H)$ is a group-like element, we denote $\ku_g$ the one dimensional vector 
space generated by $ w_g$ with
left $H$-comodule given by
$$\lambda: \ku_g \to H\otk \ku_g, \quad\lambda(w_g)=g\ot w_g.$$
A \emph{coradically graded Hopf algebra } $H=\oplus_{i=0}^m H(i)$ is a Hopf algebra $H$ that is a graded   algebra
 and a graded coalgebra
such that the coradical filtration is given by $H_n=\oplus_{i=0}^n H(i)$. For references on Hopf algebra theory
see  \cite{Mo}.

If $H$ is a coradically graded Hopf algebra and $(A, \lambda)$ is a left $H$-comodule algebra, the
\emph{Loewy series } on $A$ is given by
$A_n=\lambda^{-1}(H_n\otk A)$, $n=1,\dots, m$, see \cite{M1}. The associated graded algebra $\gr A$ is again
a left $H$-comodule algebra. If the coradical $H_0$ is a Hopf subalgebra then $A_0$ is a left
$H_0$-comodule algebra. 
The comodule algebra $A$ is  $H$-\emph{simple} if it has no non-trivial  ideals $I\subseteq A$
such that $\lambda(I)\subseteq H\otk I$.

\subsection{Twisting Hopf algebras}

In this section we recall a well known procedure of deformation of a given Hopf algebra. The reader 
is refered to  \cite{Mo}. Let $H$ be a Hopf algebra . A Hopf 2-cocycle for $H$ is a convolution invertible
  map $\sigma: H\otk H\to
\ku$,  such that
\begin{align}\label{2-cocycle}
\sigma(x\_1, y\_1)\sigma(x\_2y\_2, z) &= \sigma(y\_1,
z\_1)\sigma(x, y\_2z\_2),
\\
\label{2-cocycle-unitario} \sigma(x, 1) &= \varepsilon(x) =
\sigma(1, x),
\end{align}
for all $x,y, z\in H$. There is a new Hopf
algebra structure constructed over the same coalgebra $H$ with 
product described by
\begin{equation}\label{product-twisted}
 x._{[\sigma]}y = \sigma(x_{(1)}, y_{(1)}) \sigma^{-1}(x_{(3)},
y_{(3)})\, \, x_{(2)}y_{(2)}, \qquad x,y\in H.
\end{equation}
This new Hopf algebra is denoted by  $H^{[\sigma]}$.
 If  $(A, \lambda)$ is a left
$H$-comodule algebra, then we can define a new product in $A$ by
\begin{align}\label{sigma-product} a._{\sigma}b = \sigma(a_{(-1)},
b_{(-1)})\, a_{(0)}.b_{(0)},\quad a,b\in A.
\end{align}
 We shall denote by $A_{\sigma}$ this new algebra. With the same
comodule structure, $A_{\sigma}$ is a left $H^{[\sigma]}$-comodule algebra.

Let $H$ be a  pointed coradically graded Hopf algebra with coradical $\ku G$, $G$ a finite group.
Let $\psi\in Z^2(G,\ku^{\times})$ be a 2-cocycle.
 There exists a Hopf 2-cocycle
$\sigma_\psi:H\otk H\to \ku$ such that for any homogeneous elements $x, y\in H$
\begin{align}\label{ltwist-def}
\sigma_\psi(x,y)=\begin{cases}
               \psi(x,y), & \text{if } x,y\in H(0);\\
0, &\text{otherwise.}
  \end{cases}
\end{align}
See \cite[Lemma 4.1]{GM}.

\subsection{Bicategories}

For a review on basic notions on bicategories
we refer to \cite{Be}. Any monoidal category $\ca$ gives rise to a bicategory $\cab$
with only one object. If $\ca, \Do$
are strict monoidal categories, a \emph{pseudo-functor} $(F,\xi):\cab\to \underline{\Do}$
is  a monoidal functor between the monoidal categories $\ca$ and $\Do$. If
$(F,\xi), (G,\zeta): \ca\to \Do$ are monoidal functors, a \emph{pseudo-natural transformation}
between them is a pair $(\eta_0, \eta):(F,\xi) \to (G,\zeta)$ where $\eta_0\in \Do$ is an object
and for any $X\in \ca$ natural transformations
$$\eta_X: F(X)\ot \eta_0\to \eta_0 \ot G(X),$$
such that for all $X, Y\in\ca$
\begin{equation}\label{pseudo-nat-m}
(\id_{\eta_0}\ot \zeta_{X,Y})\eta_{X\ot Y}=(\eta_X\ot \id_{G(Y)})
(\id_{F(X)}\ot \eta_{ Y})  (\xi_{X,Y}\ot \id_{\eta_0}).
\end{equation}
Given two pseudo-natural transformations $(\eta_0, \eta):(F,\xi) \to (G,\zeta)$ and $(\sigma_0, \sigma):(G,\zeta) \to (H,\chi)$
their composition is given by
\begin{equation}\label{pseudo-nat-comp}
 (\eta_0\ot\sigma_0, (\id_{\eta_0}\ot\sigma)(\eta\ot\id_{ \sigma_0})):(F,\xi) \to (H,\chi),
\end{equation}
and their tensor product is given by
\begin{equation}\label{prod-ten-pseudonat}
(F(\sigma_0)\ot\eta_0,  \psi ).
\end{equation}
where, for any $X\in  \ca$, 
$$\psi_X: F(G(X)) \ot F( \sigma_0)\ot  \eta_0 \to F(\sigma_0)\ot \eta_0\ot G(H(X)),$$
is given by the composition
$$ \psi_X= (\id_{F(\sigma_0)}\ot \eta_{H(X)})(\xi^{-1}_{\sigma_0,H(X)}\ot\id_{\eta_0})(F(\sigma)\xi_{G(X),\sigma_0}\ot\id_{\eta_0}).  $$

If $(\eta_0,\eta),(\sigma_0,\sigma):F\to G$ are pseudo-natural transformations, a \emph{modification} $\gamma:(\eta_0,\eta)\rightrightarrows(\sigma_0,\sigma)$ is a morphism $\gamma\in\Hom_{\ca}(\eta_0,\sigma_0)$ such that for all $V\in\ca$
\begin{equation}\label{modif}
(\gamma\ot\id_{G(V)})\eta_V=\sigma_V(\id_{F(V)}\ot\gamma).
\end{equation}
 Given two modifications $\gamma: (\eta_0,\eta)\rightrightarrows(\sigma_0,\sigma)$ and
$\overline{\gamma}:(\sigma_0,\sigma)\rightrightarrows(\tau_0,\tau)$
their composition is given by the composition of morphisms in $\Do$.

 $\gamma$ is an \emph{invertible modification}
if there exist another modification $\overline{\gamma}$ such that
$\gamma\circ\overline{\gamma}=\id_{\eta_0}$ and $\overline{\gamma}\circ\gamma=\id_{\sigma_0}$.

 We say that the pseudo-natural transformations
$(\eta_0, \eta), (\sigma_0, \sigma)$ are equivalent, and it is denoted by $(\eta_0, \eta) \sim(\sigma_0, \sigma)$
if there exists an invertible modification $\gamma: (\eta_0, \eta)\to  (\sigma_0, \sigma)$. 
A pair $(\eta_0, \eta)$ is a \emph{pseudo-natural isomorphism}
 if there exists another pseudo-natural transformation $(\sigma_0, \sigma)$
such that
$$(\eta_0, \eta)(\sigma_0, \sigma) \sim ( \uno_{\Do},\id_{F}), \quad
(\sigma_0, \sigma)(\eta_0, \eta)\sim (\uno_{\Do}, \id_{G}).$$
Consequently,
the object $\eta_0$ is invertible in $\Do$, that is, there exists an object
$\overline{\eta_0}\in\Do$ such that $\eta_0\ot\overline{\eta_0}\simeq \uno_{\Do}\simeq\overline{\eta_0}\ot\eta_0$.

\subsection{Hopf biGalois objects}\label{subsection:hopfbg}

 Let $H, L$ be finite-dimensional Hopf algebras. An $(H,L)$-\emph{biGalois object}
 \cite{S2}, is an algebra $A$  that  is a
left $H$-Galois extension and a right $L$-Galois extension  of the base field $\ku$  such that
the two comodule structures make it an $(H,L)$-bicomodule. Two biGalois
objects are isomorphic if there exists a bijective bicomodule morphism that is
also an algebra map. Any  $(H,L)$-biGalois object $A$ can be regarded as a left $H\otk L^{\cop}$-comodule algebra.
It follows from \cite[Corollary 8.3.10]{Mo}  that any biGalois object is  $H\otk L^{\cop}$-simple
as a left $H\otk L^{\cop}$-comodule algebra.
\medbreak

Denote by $\biga(H)$ the set of isomorphism classes of $(H,H)$-biGalois objects. It is
a group with product given by the cotensor product $\Box_H$.

\medbreak

If $A$ is an $(H,L)$-biGalois object then the functor
 \begin{equation}\label{monoidal-eq-hopf} \Fc_A: \Comod(L)\to  \Comod(H), \quad \Fc_A(X)= A\Box_L X,
 \end{equation}
for all $X\in  \Comod(L)$, has a tensor structure as follows.
If $X, Y\in \Comod(L)$ then $ \xi^A_{X,Y}:(A\Box_L X)\otk (A\Box_L Y)\to A\Box_L ( X\otk Y)$ is defined by
\begin{equation}\label{iso-gal}
\;  \xi^A_{X,Y}(a_i \ot x_i \ot b_j\ot y_j)
= a_i b_j \ot x_i\ot y_j,
\end{equation}
for any $a_i \ot x_i \in  A\Box_L X, $ $b_j\ot y_j \in A\Box_L Y$. If $A,B$ are $(H,L)$-biGalois
objects then there is a natural monoidal isomorphism between the
tensor functors $\Fc_A, \Fc_B$ if and only if $A\simeq B$ as biGalois objects.

Assume that $A$ is a $H$-biGalois object with
left $H$-comodule structure $\lambda:A\to H\otk A$. If $g\in G(H)$ is a group-like
element we can define a new  $H$-biGalois object  $A^g$ on
the same underlying algebra $A$ with unchanged right comodule structure
and a new left $H$-comodule structure given by  $\lambda^g: A^g \to H\otk A^g$,
$\lambda^g(a)=g^{-1}a\_{-1}g\ot a\_0$
for all $a\in A$.
\medbreak

Recall \cite{FMM} that  two $H$-biGalois objects $A, B$ are \emph{equivalent},
 and denote it by $A\sim
B$ if there exists an element $g\in G(H)$ such that $A^g\simeq B$ as biGalois objects.
The subgroup of $\biga(H)$ consisting of $H$-biGalois objects equivalent to $H$ is
denoted by $\inbi(H)$. This group is a normal subgroup of $\biga(H)$. We denote $\outb(H)=\biga(H)/\inbi(H)$.

\begin{teo}\cite[Thm. 4.5]{FMM}\label{pseudo-nat-e} Let $A, B\in \biga(H)$. The following statements are equivalent.
\begin{enumerate}
 \item[1.] $A\sim B$;

\item[2.]  there exists a pseudo-natural isomorphism $(\eta_0, \eta):\Fc_A \to \Fc_B$.

\end{enumerate}\qed
\end{teo}

\begin{rmk}\label{pseudo=iso-comod} Given an isomorphism $f:A^g\to B$ of bicomodule algebras,
 there is an associated pseudo-natural isomorphism
$( \eta_0, \eta^f):\Fc_A \to \Fc_B$, given by
$$\eta_0=\ku_g,\quad \eta^f_V:A\Box_H V\otk \ku_g \to \ku_g \otk B\Box_H V, $$
$$ \eta^f_V(a\ot v \ot r)=r\ot f(a)\ot v,$$
for all $a\ot v \ot r \in A\Box_H V\otk \ku_g$. Moreover, any pseudo-natural isomorphism
is of this form.
\end{rmk}

\subsection{Comodule algebras over graded Hopf algebras}
One of the goals of the paper is the classification of biGalois objects over a certain family of Hopf algebras. 
Since biGalois objects are in particular comodule algebras, we first recall some tools developed in \cite{M1} to 
study simple comodule algebras over coradically graded Hopf algebras.
\medbreak

Let  $H=\oplus_{i=0}^m H(i)$ be a coradically graded finite-dimensional Hopf algebra.
We shall also assume that $H$ is pointed; the coradical is a group algebra $H_0=\ku G$ of a finite
group $G$.
\medbreak

If $A$ is  right $H$-simple then $A_0$ is right $\ku G$-simple, \cite[Prop. 4.4]{M1}, thus
there exists a subgroup $F\subseteq G$ and a 2-cocycle $\psi\in Z^2(F,\ku^\times)$ such
that $A_0=\ku_\psi F.$ The next result is \cite[Lemma 5.4]{M2}.

\begin{lema} If $A$ is right $H$-simple there exists a  2-cocycle $\widehat{\psi}\in Z^2(G,\ku^\times)$
such that $\widehat{\psi} $ restricted to $F$ equals $\psi$ and $(\gr A)_{\sigma_{\widehat{\psi}}}$ is isomorphic to a homogeneous left
coideal subalgebra of $H^{[\sigma_{\widehat{\psi}}]}$ as a left
$H^{[\sigma_{\widehat{\psi}}]}$-comodule algebras.\qed
\end{lema}
Recall that the Hopf 2-cocycle $\sigma_{\widehat{\psi}}$ was defined in \eqref{ltwist-def}.

\section{Finite supergroup algebras}\label{spa}

Let $G$ be a finite Abelian group, $u\in G$ be an element of order 2 and $V$ a finite-dimensional
$G$-module such that $u\cdot v=-v$  for all $v\in V$.
The space $V$ has a Yetter-Drinfeld module structure
over $\ku G$ as follows. The $G$-comodule structure $\delta:V\to \ku G\otk V$ is given by
$\delta(v)=u\ot v$, for all $v\in V$.  The Nichols algebra of $V$ is the exterior algebra $\nic(V)=\wedge(V)$.
The bosonization $\wedge(V)\# \ku G$ is called in \cite{AEG} a \emph{finite supergroup algebra}
and it is denoted by $\Ac(V,u,G)$. Hereafter we shall denote
the element $v\# g$ simply by $vg$, for all $v\in V, g\in G$.

The algebra $\Ac(V,u,G)$ is generated by elements $v\in V, g\in G$ subject to relations
$$vw+wv=0, \quad gv= (g\cdot v) g, \text{ for all } v,w\in V, g\in G.$$
The coproduct and antipode are determined for all $v\in V, g\in G$ by
\begin{align*}
\Delta(v)&=v\ot 1+ u\ot v,& \Delta(g)&=g\ot g,&
\Ss(v)&=-u v,& \Ss(g)&=g^{-1}.\end{align*}

Let us explain the coproduct in a more explicit form. Suppose $\{v_1,\ldots,v_k\}$ is a basis of $V$.
Let  $t\in \Na$, and define
$$\ele_t=\{(1,\ldots,t),(t,1,\ldots,t-1),(t-1,t,1,\ldots,t-2),\ldots,(2,3,\ldots,t,1)\}\subset \mathbb{N}^{t}.$$ 
The coproduct of  $\Ac(V,u,G)$ on the element $v_1\cdots v_t g$ of the canonical basis equals
\begin{align}\label{coprod1}
\begin{split}
& v_1\cdots v_t g\ot g+u^tg\ot v_1\cdots v_t g+
\sum_{(i_1,\ldots,i_t)\in\ele_t}v_{i_1}\cdots v_{i_{t-1}}ug\ot v_{i_t}g + \\
&\sum_{(i_1,\ldots,i_t)\in\ele_t}v_{i_1}\cdots v_{i_{t-2}}u^2g\ot v_{i_{t-1}}v_{i_t}g+
\cdots+
\sum_{(i_1,\ldots,i_t)\in\ele_t}v_{i_1}u^{t-1}g\ot v_{i_2}\cdots v_{i_{t}}g.
\end{split}
\end{align}

\begin{lema}\label{iso-cop} The algebra map
$\phi: \Ac(V,u,G)\to \Ac(V,u,G)^{\cop}$ determined by
$$\phi(v)=vu, \quad \phi(g)=g,$$
is a Hopf algebra isomorphism.\qed
\end{lema}
Next, we shall compute the projective covers of simple $\Ac(V,u,G)$-comodules.
For any $g\in G$, $\ku_g$ is a simple $\Ac(V,u,G)$-comodule. Let $P_g=\wedge(V)\otk \ku_g$ 
be the left $\Ac(V,u,G)$-comodule with coaction
determined by the restriction of the coproduct.

\begin{teo}\label{th:proy-super} Let $\{v_1,\ldots,v_k\}$ be a basis of $V$. The following assertions
hold.
\begin{itemize}
\item[1.] The family $\{\ku_g: g\in G\}$ is a complete set of isomorphism classes of simple
$\Ac(V,u,G)$-comodules.
\item[2.]  The projective cover of the comodule $\ku_{u^kg}$ is $P_g$.
\item[3.]  For all $g,h\in G$, $\ku_g\ot \ku_h\simeq \ku_{gh}$ and $P_g\ot \ku_h\simeq P_{gh}$ as
$\Ac(V,u,G)$-comodules.\end{itemize}
\end{teo}

\pf Since $\Ac(V,u,G)$ is pointed, every simple comodule is one-dimensional
 and they come from group-like elements of  $\Ac(V,u,G)$. This proves (1).

Since $\Ac(V,u,G)=\oplus_{g\in G}P_g $, as left  $\Ac(V,u,G)$-comodules, 
$P_g$ is a projective comodule for any $g\in G$.
\medbreak

Let $p_g:P_g\to \ku_{u^kg}$ be the $\Ac(V,u,G)$-comodule epimorphism, given on the elements of
the canonical basis by
\begin{equation}\label{proj-cov}
 p_g(x)=\begin{cases}
w_{u^kg} & \text{ if } x=v_1\ldots v_k g,\\
0 & \text{  elsewhere.}
\end{cases}
\end{equation}

Let us prove that this projection is essential.
Let $L$ be any $\Ac(V,u,G)$-comodule together with  a comodule morphism
 $\psi:L\to P_g$ such that $p_g\circ\psi$ is an epimorphism.
Let $y\in L$ such that $p_g\circ\psi(y)=w_{u^kg}$,
then $\psi(y)=z+\alpha\; v_1\ldots v_k\ot g $ for some $z\in$ ker$(p_g)$ and
$0\neq \alpha\in \ku$.

Note that $P_g$ is the smallest subcomodule containing $z+\alpha\; v_1\ldots v_k\ot g $.  Indeed, if
$P$ is a left subcomodule of $\Ac(V,u,G)$ such that $z+\alpha\; v_1\ldots v_k\ot g \in P$, then
using the explicit description of the coproduct given by formula \eqref{coprod1},  and the fact that $z\in$ ker$(p_g)$,
one can verify that
any element of the canonical basis of $P_g$ belongs to $P$.
Since the
image of $\psi$ is a subcomodule containing $z+\alpha\; v_1\ldots v_k\ot g $,  it must be all
$P_g$. Hence $\psi$ is surjective and the map $p_g$ is essential.
We conclude that $P_g$ is the projective cover of the comodule $\ku_{u^kg}$.

Finally, for $g,h\in G$, let $\gamma:\ku_g\ot \ku_h\to \ku_{gh}$ and $\beta:P_g\ot \ku_h\to P_{gh}$ be the maps
$$\gamma(w_g\ot w_h)=w_{gh},\quad \beta(v\ot g\ot w_h)=v\ot gh$$ for all
$v\in V$. Clearly $\gamma$ and $\beta$ are comodule isomorphisms.
\epf

The following result will be needed when computing the Frobenius-Perron dimension
of certain tensor categories.
\begin{cor}\label{grot} Assume $\dim(V)=2$. For any $g\in G$ we have
$$\langle P_g\rangle=2\langle \ku_g\rangle+2\langle \ku_{ug}\rangle.$$
Here $\langle P_g\rangle$ denotes the class of $P_g$ in
the Grothendieck group of the category of finite-dimensional left
$\Ac(V,u,G)$-comodules.
\end{cor}
\pf Let $\{v,w\}$ be a basis of $V$. Recall the projection $p_g:P_g\to \ku_{g}$ described in
\eqref{proj-cov}. Since in this case $P_g$ is generated as a vector space by
$\{vw\ot g, v\ot g, w\ot g, 1\ot g\}$, the kernel of $p_g$ is generated as a vector space by
$\{ v\ot g, w\ot g, 1\ot g\}$.
Define $f:\ker(p_g)\to \ku_{ug}$  the $\Ac(V,u,G)$-comodule epimorphism  by
$$f(x)=\begin{cases}
w_{ug} & \text{ if } x=w\ot g,\\
0 & \text{ elsewhere.}\end{cases}$$

Let $f_1:\ker(f)\to \ku_{ug}$ be the $\Ac(V,u,G)$-comodule epimorphism given by
$$f_1(x)=\begin{cases}
w_{ug} & \text{ if } x=v\ot g,\\
0 & \text{ elsewhere.}\end{cases}$$

We have a composition series for $P_g$  given by
$$P_g\supseteq \ker(p_g)\supseteq \ker (f) \supseteq \ker(f_1)\supseteq 0,$$ and satisfies
\begin{align*}
P_g/\ker(p_g)&\simeq \ku_g,& \ker(p_g)/\ker(f)&\simeq \ku_{ug},\\
\ker(f)/\ker(f_1)&\simeq \ku_{ug},& \ker(f_1)&\simeq \ku_g.\end{align*}

\epf

\subsection{The tensor product $\Ac(V,u,G)\otk\Ac(V,u,G)^{\cop}$}\label{tp-spa} Let  $G_1, G_2$ be  finite 
Abelian groups and
$u_i\in G_i$ be central elements of order 2. For $i=1,2$ let $V_i$ be finite-dimensional
$G_i$-modules,  such that $u_i$ acts in $V_i$ as $-1$.

Define  $\Ac(V_1,V_2,u_1,u_2,G_1,G_2)=\Ac(V_1,u_1,G_1)\otk \Ac(V_2,u_2,G_2)$
with the tensor product Hopf algebra structure. For simplicity, we shall denote
$$\Bc(V,u,G)=\Ac(V,V,u,u,G,G).$$
Observe that $\Bc(V,u,G)$ is a  coradically graded Hopf algebra.

If we denote $D=G_1\times G_2$, then both vector spaces $V_1, V_2$ are $D$-modules by setting
$$ (g,h)\cdot v_1= g\cdot v_1, \quad   (g,h)\cdot v_2= h\cdot v_2, \quad (g,h)\in D, v_i \in V_i;i=1,2.$$
The algebra $\Ac(V_1,V_2,u_1,u_2,G_1,G_2)$
is generated by elements
$ V_1, V_2, D$
subject to relations
$$  v_1w_1+ w_1v_1=0,\;\;  v_2w_2+ w_2v_2=0,\; \; v_1v_2=v_2v_1,$$
$$g v_1=  (g\cdot v_1) g,\quad   gv_2= (g\cdot v_2) g,$$
for all $g\in D$, $v_i, w_i\in V_i$, $i=1,2$. The Hopf algebra structure
is determined for all $(g_1,g_2)\in G$, $v_i\in V_i$, $i=1,2$ by
$$ \Delta(v_1)= v_1\ot 1+ (u_1,1)\ot v_1, \;\;\; \Delta(v_2)=v_2\ot 1+ (1,u_2)\ot v_2,$$
$$ \Delta(g_1,g_2)=(g_1,g_2)\ot (g_1,g_2).$$

\medbreak

We shall define certain families of Hopf algebras that are cocycle deformations of
$\Bc(V,u,G)$. Let $(V_1,V_2,u_1,u_2,G_1,G_2)$ be a data as
above. Set $V=V_1\oplus V_2$. Define
$\Hc(V_1,V_2,u_1,u_2,G_1,G_2)= \wedge(V)\otk \ku D$ with product
determined by
$$vw+wv=0, \quad g v= (g\cdot v) g, \text{ for any }\,  v,w\in V_1\oplus V_2, g\in D,$$
and coproduct determined by
$$ \Delta(v_1)= v_1\ot 1+ (u_1,1)\ot v_1, \;\;\; \Delta(v_2)=v_2\ot 1+ (1,u_2)\ot v_2,$$
 for any $v_i\in V_i$, $i=1,2$.

\begin{lema}\label{twisting-supergroup}\cite[Prop. 6.2]{M2} Let be $H= \Ac(V_1,V_2,u_1,u_2,G_1,G_2)$,
$\psi\in Z^2(D,\ku^{\times})$ and $\sigma_\psi:H
\otk H\to \ku$ the Hopf 2-cocyle defined in \eqref{ltwist-def}. Denote
$$\xi=\psi((u_1,1),(1,u_2))\psi((1,u_2),(u_1,1))^{-1}.$$ Then
\begin{itemize}
 \item[(i)] if $\xi=1$ we have $H^{[\sigma_\psi]}\simeq \Ac(V_1,V_2,u_1,u_2,G_1,G_2)$;
 \item[(ii)] if $\xi=-1$ then $H^{[\sigma_\psi]}\simeq \Hc(V_1,V_2,u_1,u_2,G_1,G_2)$.
\end{itemize}\qed
\end{lema}

\section{The classification of Hopf BiGalois objects over $\Ac(V,u,G)$}\label{spa-hb}

 In this section we shall
 present a classification of BiGalois objects over the supergroup algebras.
The idea to achieve this classification for an arbitrary Hopf algebra $H$ is the following. 
 Any biGalois object over $H$ is an $H\otk H^{\cop}$-simple left  
 $H\otk H^{\cop}$-comodule algebra with trivial coinvariants.  Any such 
 $H\otk H^{\cop}$-comodule algebra is a \emph{lifting} of a 2-cocycle deformation 
 of a homogeneous left coideal subalgebra inside certain 
 a twisting of the Hopf algebra $H\otk H^{\cop}$. Since biGalois objects have dimension equal to the 
 dimension of $H$, we can then detect
the biGalois objects.
 
 \medbreak 
 
Let $G$ be a finite Abelian group, $u\in G$ be an element of order 2 and $V$ a finite-dimensional
$G$-module such that $u\cdot v=-v$  for all $v\in V$. 

First we classify all  $\Ac(V,u,G)\otk \Ac(V,u,G)^{\cop}$-simple
left comodule algebras with trivial coinvariants. Hopf biGalois objects
over $\Ac(V,u,G)$ are inside this family.

\subsection{Simple comodule algebras over $\Bc(V,u,G)$}

We recall the description of all $\Bc(V,u,G)$-simple left comodule algebras
presented in \cite{M2}. \medbreak

For a given finite-dimensional coradically graded Hopf algebra $H$, the idea to classify simple left $H$-comodule algebras, is roughly the
following. If $A$ is a  $H$-simple left comodule algebra the graded algebra $\gr A$,
with respect to the Loewy filtration, is also $H$-simple. A twisting of  $\gr A$, by
a certain Hopf 2-cocycle $\sigma$, is isomorphic to an homogeneous coideal
subalgebra inside $H^{[\sigma]}$. Then, one has to classify homogeneous coideal
subalgebras inside $H^{[\sigma]}$.
At last, one has to compute all
\emph{liftings} of $\gr A$, that is , $H$-comodule algebras $A$ such that $\gr A$ is a
twisting of a coideal subalgebra inside $H^{[\sigma]}$.

\begin{defi} A collection $(W^1,W^2,W^3, \beta, F, \psi)$
is  \emph{compatible  } with the triple $(V,u,G)$ if
\begin{itemize}

 \item $W^1, W^2\subseteq V$,  $W^3\subseteq V\oplus V$ are subspaces 
such that $W^3 \cap W^1\oplus W^2= 0$,
$W^3 \cap V \oplus \{0\} =0=W^3 \cap  \{0\} \oplus V$;
 \item $F\subseteq G\times G$ is  a subgroup that leaves invariant all  subspaces $W^i$, $i=1,2,3$;
 \item if $W^3\neq 0$ then $(u, u)\in F$;

 \item denote $W=W^1\oplus W^2\oplus W^3$. Then $\beta:W\times W\to \ku$
is a bilinear form stable under the action of $F$, such that
$$\beta(w_1,w_2)=-\beta(w_2,w_1),\; \beta(w_1,w_3)=\beta(w_3,w_1),\; \beta(w_2,w_3)=-\beta(w_3,w_2),
$$
for all $w_i\in W^i$, $i=1,2,3$, and $\beta$ restricted to $W^i\times W^i$ is symmetric
for any $i=1,2,3$;
\item if $(u, u)\notin F$ then  $\beta$ restricted to $W^1\times W^2$ and $ W^2\times  W^3$ is null;
\item  $\psi\in H^2(F,\ku^{\times})$.
\end{itemize}

\end{defi}

If $(W^1,W^2,W^3, \beta, F, \psi)$ is compatible with $(V,u,G)$ the left $\Bc(V,u,G)$-comodule algebra
$\kc(W, \beta, F, \psi)$  is defined as follows. The algebra $\kc(W, \beta, F, \psi)$  is generated by $W$ and  $\{e_f: f\in F\}$,
subject to relations
$$ e_f e_h=\psi(f,h)\, e_{fh}, \quad e_f w= (f\cdot w)  e_f,$$
$$w_iw_j+w_jw_i=\beta(w_i,w_j) 1,\quad w_i\in W^i, w_j\in W^j,$$
for any $ (i,j)\in
\{(1,1), (2,2), (1,3), (3,3) \},$ and relations
$$ w_2 w_3 - w_3 w_2= \beta(w_2,w_3)\, e_{(u, u)} , \text{ for any  } w_2\in W^2, w_3\in W^3,$$
$$ w_1 w_2 - w_2 w_1= \beta(w_1,w_2)\, e_{(u, u)}, \text{ for any  } w_1\in W^1, w_2\in W^2.$$
The left coaction $\delta: \kc(W, \beta, F, \psi)\to \Bc(V,u,G)\otk \kc(W, \beta, F, \psi)$ is defined on the generators
$$\delta(e_f)=f\ot e_f,\quad \delta(v,w)=v\ot 1 + w (u,u)\ot e_{(u, u)} +  (u,1) \ot  (v,w),$$
$$\delta(w_2)= w_2\ot 1 + (1,u)\ot w_2, \quad
\delta(w_1)= w_1\ot 1 + (u,1)\ot w_1, $$
for any $f\in F, w_1\in W^1$, $w_2\in W^2,$ $ (v,w)\in W^3.$ This family of comodule algebras was 
introduced in \cite{M2} to classify certain module categories.

\begin{defi}  If  $(W^1,W^2,W^3, \beta, F, \psi)$ is a compatible data with $(V,u,G)$
such that $W^1=W^2=0$ we shall denote
$\ele(W,\beta, F, \psi)=\kc(W, \beta, F, \psi)$.
\end{defi}

The following result is \cite[Prop. 7.4, Thm. 7.10]{M2}.
\begin{teo}\label{class-simple-ca}  The following assertions hold.
\begin{itemize}
\item[1.] $\dim \kc(W, \beta, F, \psi)= \dim W |F|.$
 \item[2.] The algebra $\kc(W, \beta, F, \psi)$  is a $\Bc(V,u,G)$-simple
 left comodule algebra with trivial coinvariants.
\end{itemize}
Moreover, any  $\Bc(V,u,G)$-simple
 left $\Bc(V,u,G)$-comodule algebra with trivial coinvariants is isomorphic to
one $\kc(W, \beta, F, \psi)$ for some compatible data $(W, \beta, F, \psi)$.\qed
\end{teo}

For later use, we shall give explicitly the left and right coactions on
the algebra $\ele(W,\beta, \psi)$. Any left $\Bc(V,u,G)$-comodule is a
$\Ac(V,u,G)$-bicomodule where the  right coaction is obtained using the canonical projection
$$\epsilon\ot\id:\Bc(V,u,G)=\Ac(V,u,G)\otk \Ac(V,u,G) \twoheadrightarrow \Ac(V,u,G),$$
composed with the isomorphism
$\phi: \Ac(V,u,G)\to \Ac(V,u,G)^{\cop}$ given in Lemma \ref{iso-cop}.
\medbreak

The $\Ac(V,u,G)$-bicomodule structure on $\ele(W,\beta, F,\psi)$ is given by
the left and right actions $\lambda:\ele(W,\beta, F,\psi ) \to \Ac(V,u,G)\otk \ele(W,\beta,F,\psi ),$
$\rho:\ele(W,\beta,F, \psi ) \to \ele(W,\beta,F, \psi) \otk \Ac(V,u,G)$ determined by
\begin{equation}\label{r-l-coact}
 \lambda(v,w)=v\ot 1 + u\ot (v,w), \quad \rho(v,w)= e_{(u,u)}\ot w+ (v,w)\ot 1,
\end{equation}
$$\lambda(e_{(g,f)})=g\ot e_{(g,f)}, \quad \rho(e_{(g,f)})=e_{(g,f)}\ot f,$$
for all $(g, f)\in F$, $(v,w)\in W$.

\begin{lema} If  $F\subseteq G\times G$ is a subgroup such that $(u,u)\in F$,
$|F|=|G|$, $F\cap G\times \{1\}=\{1\}= F\cap \{1\}\times G$ and
$W\subseteq V\oplus V$ is a subspace stable under the action of $F$ such that  $\dim W=\dim V$,
$W \cap V \oplus 0 =0=W \cap 0 \oplus V$; then the comodule algebras $\ele(W,\beta, F, \psi)$ are $\Ac(V,u,G)$-biGalois
objects.
\end{lema}
 \pf We shall prove that the algebra $\ele(W,\beta, F, \psi)$ is a Hopf-Galois object from the left. The proof that it is
 Hopf-Galois from the right is similar. The conditions on the subgroup $F$ assures that the comodule algebra 
 $\ele(W,\beta, F, \psi)$ has trivial coinvariants.
 We must show that the canonical map
 $$can: \ele(W,\beta, F, \psi)\otk \ele(W,\beta, F, \psi)\to \Ac(V,u,G)\otk \ele(W,\beta, F, \psi),  $$
 $$can(a\ot b)= a\_{-1}\ot a\_0 b, $$
 is an isomorphism. By Theorem \ref{class-simple-ca} (1) the dimension of $\ele(W,\beta, F, \psi)$
 equals the dimension of $\Ac(V,u,G)$, hence it is enough to prove that $can$ is surjective. 
 The map $can$ is surjective if for any algebra generator $a\in \Ac(V,u,G)$ there exists an element
 $z\in \ele(W,\beta, F, \psi)\otk \ele(W,\beta, F, \psi)$ such that $can(z)=a\ot 1$. 
 
 Since $|F|=|G|$, for any $g \in G$ there exists $f\in G$ such that $(g,f)\in F$. Then
 $$can(e_{(g,f)}\ot e_{(g^{-1},f^{-1})} )=g\ot 1.$$
 Since $\dim W=\dim V$, for any $v\in V$ there exists $w\in V$ such that $(v,w)\in W$. Then, since $(u,u)\in F$
 $$ can((v,w)\ot 1- e_{(u,u)}\ot e_{(u,u)} (v,w))= v\ot 1$$
 \epf

\subsection{Hopf BiGalois objects over $\Ac(V,u,G)$ }
We shall use the description of $\Bc(V,u,G)$-simple left comodule algebras given in the previous
section to classify $\Ac(V,u,G)$-Hopf BiGalois objects.

\begin{teo}\label{class-hopf-bigal} Any $\Ac(V,u,G)$-biGalois object is isomorphic to an algebra
of the form  $\ele(W,\beta, F, \psi)$, where
\begin{itemize}
\item $F\subseteq G\times G$ is a subgroup such that
$F\cap G\times \{1\}=\{1\}= F\cap \{1\}\times G$, $|F|=|G|$, $(u,u)\in F$;
 \item $W\subseteq V\oplus V$ is a subspace stable under the action of $F$ such that  $\dim W=\dim V$,
$W \cap V \oplus 0 =0=W \cap 0 \oplus V$;
\item $\beta:W\times W\to \ku$ is a $F$-invariant symmetric bilinear form;
\item and $\psi\in H^2(F, \ku^{\times})$ is a 2-cocycle.

\end{itemize}
\pf

Let $A$ be a $\Ac(V,u,G)$-biGalois object. We have that $A$ is a  $\Bc(V,u,G)$-simple left
$\Bc(V,u,G)$-comodule algebra with trivial coinvariants. This implies that there exists
a compatible data $(W^1,W^2,W^3, \beta, F, \psi)$
such that $A\simeq \kc(W, \beta, F, \psi)$. Since the coinvariants of $A$ are trivial, 
$W^1=W^2=0$ and $W=W^3$. The conditions stated on $F$ and $W$ must be satisfied since
 the coinvariants of $A$ are trivial and $\dim A=\dim H$.
\epf

\end{teo}

Now, we shall give an alternative description of
compatible data $(W, \beta, F, \psi)$ such that the comodule algebra $\ele(W,\beta, F,\psi )$
is a biGalois object. \smallbreak

A collection $(T,\beta, \alpha, \psi)$  will  be also called a \emph{compatible data} if:
\begin{itemize}
 \item  $\alpha: G\to G$ is a group isomorphism such that $\alpha(u)=u$;
 \item $T:V\to V$ is a linear automorphism such that
$$T(g\cdot v)=\alpha(g)\cdot T(v),\quad v\in V, g\in G;$$
\item $\beta:V\times V\to \ku$ is a symmetric $G$-invariant bilinear form;
\item $\psi\in H^2(G, \ku^{\times})$ is a 2-cocycle.
\end{itemize}

\begin{lema}\label{b-c1} There is a bijective correspondence between the set of
compatible data $(T,\beta, \alpha, \psi)$ and collections $(W,\beta, F, \psi)$ such that they satisfy the
conditions of Theorem \ref{class-hopf-bigal}.
\end{lema}
\pf If $(T,\beta, \alpha, \psi)$ is a compatible data define $(W,\widehat{\beta}, F, \widehat{\psi})$ as follows:
  $$W=\{(T(v),v): v\in V\},\quad F=\{(\alpha(g),g): g\in G\}.$$
The bilinear form $\widehat{\beta}$ and the 2-cocycle $ \widehat{\psi}$ are defined as
$$\widehat{\beta}((T(v),v),( T(w),w))=\beta(v,w), \quad
\widehat{\psi}((\alpha(g),g),(\alpha(f),f))=\psi(g,f),$$
for all $v, w\in V,$ $g,f\in G$. Let  $(W, \beta, F, \psi)$ be a compatible data satisfying
conditions of Theorem \ref{class-hopf-bigal}. If $(x,g)\in F$, since $F\cap G\times\{1\} = \{1\}$,
then $x$ is uniquely determined by the element $g$. 
So we can denote $x=\alpha(g)$. Since $|F|=|G|$ the function $\alpha$ is defined 
for any $g\in G$. Also, since $F\cap \{1\}\times G= \{1\}$, the map $\alpha $ is injective. The fact that 
$|F|=|G|$ implies that it is bijective.
Since $F$ is a group, $\alpha$ is a group homomorphism, hence it is a group isomorphism. The definition of the linear isomorphism $T$ is 
analogous. Both constructions are one the inverse of the other.
\epf

\begin{defi} If $(T,\beta, \alpha, \psi)$ is a compatible data  denote
$\ele(T,\beta, \alpha, \psi)$  the algebra  $\ele(W,\beta, F, \psi)$ where
the collection $(W,\beta, F, \psi)$ is the associated data to $(T,\beta, \alpha, \psi)$
under the correspondence of Lemma \ref{b-c1}. If $(T,\beta, \alpha, \psi)$, $(T',\beta', \alpha', \psi')$
are compatible data, define
$$(T,\beta, \alpha, \psi)\bullet (T',\beta', \alpha', \psi') = (T\circ T',\beta\circ T'+\beta',\alpha\circ \alpha',\psi \psi'). $$
If $g\in G$ define $T_g:V\to V$ the isomorphism $T_g(v)=g\cdot v$ for all $v\in V$. Then
$(T_g, 0,\id,1)$ is a compatible data for all $g\in G$.
\end{defi}

 \begin{lema} Let $(T,\beta, \alpha, \psi)$, $(T',\beta', \alpha', \psi')$ be compatible data.
\begin{itemize}
 \item[1.] The collection $(T\circ T',\beta\circ T'+\beta',\alpha\circ \alpha',\psi \psi') $ is a compatible data.

 \item[2.] The set of compatible data with product
\begin{equation}\label{product-in-big}
(T,\beta, \alpha, \psi)\bullet (T',\beta', \alpha', \psi') = (T\circ T',\beta\circ T'+\beta',\alpha\circ \alpha',\psi \psi')
\end{equation}
is a group with identity $(\Id,0,\id, 1)$.
\end{itemize}

 \end{lema}
\pf 1. Straightforward.
\medbreak

2. For any compatible data $(T,\beta, \alpha, \psi)$
the collection $(T^{-1},-\beta\circ T ^{-1}, \alpha^{-1}, \psi^{-1})$
is again a compatible data and it is the inverse of $(T,\beta, \alpha, \psi)$.
\epf

\begin{defi} Define  the group $\qs(V,u,G)$ as the quotient of the set
of compatible data $(T,\beta, \alpha, \psi)$ with product
described in \eqref{product-in-big} modulo the normal subgroup of order two generated by
the element $(T_u, 0, \id, 1)$.
\end{defi}

 The set of compatible data
$\{(T_g, 0,\id,1): g\in G \}$ is a normal subgroup of $\qs(V,u,G)$.  The quotient group
$\qs(V,u,G)/\{(T_g, 0,\id,1): g\in G  \} $ is denoted by $ \os(V,u,G)$.

\begin{prop}\label{bi-inn} Let $(T,\beta, \alpha, \psi)$, $(T',\beta', \alpha', \psi')$  be compatible data.
The following assertions hold.
\begin{itemize}
 \item[1.] There is an isomorphism $\ele(T,\beta, \alpha, \psi)\simeq \ele(T',\beta', \alpha', \psi')$ of
biGalois objects if and only 
$$(T,\beta, \alpha, \psi)=(T',\beta', \alpha', \psi')
\text{ or }
(T_u\circ T,\beta, \alpha, \psi)= (T',\beta', \alpha', \psi').$$

\item[2.] $\ele(T,\beta, \alpha, \psi)\in \inbi(\Ac(V,u,G))$ if and only if
$(T,\beta, \alpha, \psi)= (T_g, 0,\id,1)$  for some $g\in G$.
 \item[3.] There is an isomorphism of $\Bc(V,u,G)$-comodule algebras
$$\ele(T,\beta, \alpha, \psi)\Box_{\Ac(V,u,G)} \ele(T',\beta', \alpha', \psi')\simeq
\ele(T\circ T',\beta\circ T'+\beta',\alpha\circ \alpha',\psi \psi').$$
\end{itemize}
\end{prop}
\pf 1. Let $f: \ele(T,\beta, \alpha, \psi)\to \ele (T',\beta', \alpha', \psi')$ be a
$\Bc(V,u,G)$-comodule algebra isomorphism. This implies that for any
$g\in G$ we have $f (e_{(g,\alpha(g))})= \chi_g\, e_{(g,\alpha(g))}$ for some
$\chi_g\in \ku$. Whence $\psi=\psi'$ in
$H^2(G, \ku^{\times })$. Since $e_{(u,u)}^2=1$ we have $\chi_u= \pm1$.

Denote by  $(W,\beta, \psi)$, $(W',\beta', \psi')$ the collections associated to
the compatible data $(T,\beta, \alpha, \psi)$ and $(T',\beta', \alpha', \psi')$, respectively, under
the correspondence of Lemma \ref{b-c1}.
Follows straightforward that $f(W)=W'$. If $f(x,y)=(x',y')$ for
$(x,y)\in W$ then, since
$f$ is a $\Bc(V,u,G)$-comodule map, the element
$$x'\ot 1+  y'(u,u)\ot e_{(u,u)} + (u,1)\ot (x',y')$$
is equal to
$$x\ot 1+ \chi_u \, y(u,u)\ot e_{(u,u)} + (u,1)\ot (x',y'). $$
Thus $f(x,y)=(x,\chi_u \,y)$. If $\chi_u=1$ both collections $(W,\beta, \psi)$, $(W',\beta', \psi')$
are equal. If $\chi_u=-1$ then $(T_u\circ T,\beta, \alpha, \psi)= (T',\beta', \alpha', \psi')$.

2.  Recall the definition of $\inbi(H)$ given in Section \ref{subsection:hopfbg}.
It follows directly from (1) and the definition of $\inbi(\Ac(V,u,G))$.

3. Define the algebra map
$$\vartheta: \ele(T\circ T',\beta\circ T'+\beta',\alpha\circ \alpha',\psi \psi')
\to \ele(T,\beta, \alpha, \psi)\Box_{\Ac(V,u,G)} \ele(T',\beta', \alpha', \psi')$$ as follows. If
$g\in G, v\in V$ then
$$\vartheta(T\circ T'(v),v)= (T\circ T'((v), T'(v))\ot 1 + e_{(u,u)}\ot (T'(v),v), $$
$$\vartheta(e_{(\alpha\circ \alpha'(g),g)})=
e_{(\alpha\circ \alpha'(g),\alpha'(g))}\ot e_{(\alpha'(g),g)}.$$
It follows by a straightforward calculation that the image of $\vartheta$ is inside
$\ele(T,\beta, \alpha, \psi)\Box_{\Ac(V,u,G)} \ele(T',\beta', \alpha', \psi')$.
The map
$\vartheta$ is an injective algebra map. Since both algebras have the same dimension,
$\vartheta$ is an isomorphism.
\epf

\begin{rmk} The proof of Part (1) of Proposition  \ref{bi-inn}
 gives a description of the possible bicomodule algebra isomorphisms between two
biGalois objects. This fact will be used later.
\end{rmk}

\begin{cor} There are group isomorphisms
$$\qs(V,u,G)\simeq \biga(\Ac(V,u,G)), \quad \os(V,u,G)\simeq \outb(\Ac(V,u,G)).$$\qed
\end{cor}

\begin{rmk} As a consequence of \cite[Corollary 4.9]{FMM} and Proposition \ref{bi-inn} there is an exact sequence of groups
$$0\to G/<u>\to \qs(V,u,G) \to \brp(\Rep(\Ac(V,u,G)).$$
\end{rmk}

\begin{lema}\label{lm:iso-esp} Let  $(T,\beta,\alpha,\psi)$ be a compatible data and $g\in G$.
Then there is an ismorphism
$\ele(T,\beta,\alpha,\psi)\Box_{\Ac(V,u,G)}\ku_g\simeq\ku_{\alpha(g)}$ of left
$\Ac(V,u,G)$-comodules.
\end{lema}
\begin{proof} If $a\ot r\in\ele\Box_H\ku_g$ then $\rho(a)=a\ot g$, hence
 $$\rho(ae_{(\alpha(g^{-1}),g^{-1})})=(a\ot g)(e_{(\alpha(g^{-1}),g^{-1})}\ot g^{-1})
=ae_{(\alpha(g^{-1}),g^{-1})}\ot 1,$$
therefore $ae_{(\alpha(g^{-1}),g^{-1})}\in\ku 1=\ele(T,\beta,\alpha,\psi)^{\co \Ac(V,u,G)}$,
 and $a=\zeta e_{(\alpha(g),g)}$ for some $\zeta \in \ku$. \end{proof}

\subsection{A concrete example of biGalois extensions}\label{sec:exam-bigal}

Assume  $V$ is the 2-dimensional vector space generated by
$\{v_1, v_2\}$ and $G=C_2=<u>$ the cyclic group with two elements. Then, $V$ is a
$C_2$-module with action determined by declaring $u\cdot v_i=-v_i$ for $i=1,2$.

For any $\xi\in \ku$ define $T_\xi:V\to V$ the linear map
$$T_\xi(v_1)=v_1, \quad T_\xi(v_2)=\xi v_1 - v_2.$$

 By Lemma \ref{b-c1}, the compatible data $(T_\xi,0,\id,1)$
gives rise to a $\Ac(V,u,C_2)$-biGalois extension that we denote by $\textbf{U}_\xi$.
From Proposition \ref{bi-inn} (3) it follows that $\textbf{U}_\xi$ has order two, that is,
there is a bicomodule algebra isomorphism $\textbf{U}_\xi\Box_H \textbf{U}_\xi\simeq H.$

\section{Crossed product tensor categories}\label{section:crossed}

In this section $\ca$ will denote a strict finite tensor category.
We recall the definition of crossed system of a finite group $\Gamma$ on the tensor category
$\ca$ introduced in
\cite{Ga1} and the associated $\Gamma$-graded extension of $\ca$.

\begin{defi}\cite{Ga1}\label{def-cross-syst} Let $\Gamma$ be a finite group. \emph{A crossed system of $\Gamma$ over $\ca$}
is a collection $\Sigma=((a_*,\xi^a),(U_{a,b},\sigma^{a,b}),\gamma_{a,b,c})_{a,b,c\in \Gamma}$
consisting of
\begin{itemize}

\item Monoidal autoequivalences $(a_*, \xi^a):\ca\to\ca$
 where $\xi^a_{X,Y}:a_*(X\ot Y)\to a_*(X)\ot a_*(Y)$
 is the monoidal structure for $X,Y\in\ca$. We also require that $a_*(\uno)=\uno$;
\item Objects $U_{a,b}\in\ca$ and for any $X\in \ca$ natural isomorphisms
$$\sigma^{a,b}_X:a_*b_*(X)\ot U_{a,b}\to U_{a,b}\ot (ab)_*X, \quad X\in\ca;$$
\item isomorphisms $\gamma_{a,b,c}:a_*(U_{b,c})\ot U_{a,bc}\to U_{a,b}\ot U_{ab,c}$;
\end{itemize}
such that for all $a,b,c\in \Gamma$, $X,Y\in\ca$:
\begin{equation}\label{crossed-syst11} \sigma_\uno^{a,b}=\id_{U_{a,b}},\quad 1_*=\Id_{\ca},\quad
(U_{1,a},\sigma^{1,a})=(\uno,\id_{a_*})= (U_{a,1},\sigma^{a,1}),
\end{equation}
\begin{equation}\label{crossed-syst12}
\gamma_{a,1,b}=\gamma_{1,a,b}=\gamma_{a,b,1}=\id_{U_{a,b}},
\end{equation}

\begin{equation}\label{crossed-syst13}
(\id_{U_{a,b}}\ot \xi^{ab}_{X,Y})\sigma^{a,b}_{X\ot Y}=\end{equation}
$$(\sigma^{a,b}_X\ot \id_{(ab)_*(Y)})(\id_{a_*b_*(X)}\ot \sigma^{a,b}_Y)(\xi^a_{b_*X,b_*Y}a_*(\xi^b_{X,Y})\ot \id_{U_{a,b}}),$$
\begin{equation}\label{crossed-syst14}(\gamma_{a,b,c}\ot \id_{(abc)_*(X)})(\id_{a_*(U_{b,c})}\ot\sigma_X^{a,bc})
(\xi^a_{U_{b,c},(bc)_*(X)}\circ a_*(\sigma_X^{b,c})\ot \id_{U_{a,bc}})
=\end{equation}
$$=(\id_{U_{a,b}}\ot\sigma^{ab,c}_{X})
(\sigma^{a,b}_{c_*X}\ot \id_{U_{ab,c}})(\id_{a_*b_*c_*(X)}\ot\gamma_{a,b,c})
 (\xi^a_{b_*c_*(X),U_{bc}}\ot \id_{U_{a,bc}}).$$

\end{defi}

\begin{rmk} 1. Condition \eqref{crossed-syst13} of  Definition (\ref{def-cross-syst})
 implies that $(U_{a,b},\sigma^{a,b})$ is a pseudo-natural isomorphism in the bicategory $\cab$
with only one object.
In particular the object $U_{a,b}$ is invertible in $\ca$ with inverse $\overline{U_{a,b}}$.

2. Condition \eqref{crossed-syst14} implies that $\gamma_{a,b,c}$ is an invertible modification in the same bicategory.
\end{rmk}

\begin{defi}  A crossed system $\Sigma =((a_*,\xi^a),(U_{a,b},\sigma^{a,b}),\gamma_{a,b,c})_{a,b,c\in \Gamma}$
is a \emph{coherent outer $\Gamma$-action}
on $\ca$ if for all $a,b,c,d\in \Gamma$
\begin{equation}\label{eq-acts}(\gamma_{a,b,c}\ot \id_{U_{abc,d}})(\id_{a_*(U_{b,c})}\ot \gamma_{a,bc,d})(\xi^a_{U_{bc},U_{bc,d}} a_*(\gamma_{b,c,d})\ot \id_{U_{a,bcd}})=\end{equation}
$$(\id_{U_{a,b}}\ot\gamma_{ab,c,d})
(\sigma^{a,b}_{U_{cd}}\ot \id_{U_{ab,cd}})(\id_{a_*b_*(U_{cd})}\ot
\gamma_{a,b,cd})(\xi^a_{b_*(U_{c,d}),U_{b,cd}}\ot \id_{U_{a,bcd}}).$$
In this case, we say that $\Gamma$ \emph{acts} on the category $\ca$.
\end{defi}

If $\Gamma$ acts on $\ca$ via a crossed system $\Sigma$, then
the \emph{$\Gamma$-crossed product tensor category}, introduced in \cite{Ga1}, associated to this action
is $\ca(\Sigma)$, where $\ca(\Sigma)=\oplus_{a\in \Gamma}\ca_a$ as Abelian categories and
$\ca_a=\ca$ for all $a\in \Gamma$. Denote by $[V,a]$ the object $V\in\ca_a$. Morphisms
 from $\oplus_{a\in \Gamma}[V_a,a]$ to $\oplus_{a\in \Gamma}[W_a,a]$ are given by $\oplus_{a\in \Gamma}[f_a,a]$
where $f_a:V_a\to W_a$ is a morphism in $\ca$ for all $a\in \Gamma$.

\begin{teo}\cite[Sec. 3.3]{Ga1}\label{ten-prod} $\ca(\Sigma)$ is a tensor category with tensor product
 $\ot:\ca(\Sigma)\times \ca(\Sigma)\to\ca(\Sigma)$ defined by
\begin{align}
[V,a]\ot[W,b]&=[V\ot a_*(W)\ot U_{a,b},ab] \text{ on objects},\\
[f,a]\ot[g,b]&=[f\ot a_*(g)\ot \id_{U_{a,b}},ab] \text{ on morphisms},
\end{align}
with unit object $[\uno_{\ca}, 1]$, and associativity constraints given by
\begin{align}
\alpha_{[V,a][W,b][Z,c]}&=(\id_{V\ot a_*W}\ot\sigma^{a,b}_Z\ot \id_{U_{ab,c}})(\id_{V\ot a_*W\ot a_*b_*Z}\ot\gamma_{a,b,c})\circ\end{align}
$$(\id_{V\ot a_*W}\ot \xi^a_{b_*Z,U_{b,c}}\ot \id_{U_{a,bc}})(\id_V\ot\xi^a_{W,b_*Z\ot U_{b,c}}\ot \id_{U_{a,bc}}).$$

The dual objects are given by $$([V,1])^*=[V^*,1]\text{ and }
([\uno,a])^*=[\overline{U_{a,a^{-1}}},a^{-1}].$$\qed
\end{teo}

The next result explains when, for two coherent outer actions
$\Sigma, \Sigma'$,  the tensor categories $\ca(\Sigma)$, $\ca(\Sigma')$
are monoidally equivalent.

\begin{teo}\cite[Th. 4.1]{Ga1}\label{th:eq-mon} Let $\Sigma=((a_*,\varrho^a),(U_{a,b},\sigma^{a,b}),\gamma_{a,b,c})_{a,b,c\in \Gamma},
\Sigma'=((a',\zeta^a),(U'_{a,b},\tau^{a,b}),\gamma'_{a,b,c})_{a,b,c\in \Gamma}$ be two
coherent outer $\Gamma$-actions over $\ca$.
 Any monoidal equivalence
$F:\ca(\Sigma)\to \ca(\Sigma')$ comes from a collection

 $((H,\xi),f,(\theta_a,\beta^a), \chi_{a,b})_{a,b\in\Gamma}$ where:
\begin{itemize}
\item $(H,\xi):\ca\to\ca$ is a monoidal equivalence;
\item $f:\Gamma\to\Gamma$ is a group isomorphism;
\item for any $a\in \Gamma$ the pair $(\theta_a,\beta^a):H\circ a_*\to f(a)'\circ H$ is a pseudo-natural isomorphism
 such that $(\theta_1,\beta^1)=(\uno,\id)$;
\item $\chi_{a,b}:H(U_{a,b})\ot\theta_{ab}\to \theta_a\ot f(a)'(\theta_b)\ot U'_{f(a),f(b)}$ is an invertible morphism in $\ca$ such that $\chi_{a,1}=\chi_{1,a}=\id_{\theta_a}$ and
    \begin{equation}\label{eq:chi-general}
    p_V(\id_{H(a_*b_*(V))}\ot\chi_{a,b})=(\chi_{a,b}\ot\id_{f(ab)'(H(V))})q_V, \quad  V\in\ca,\end{equation}
    where
    \begin{align*}
    p_V&=(\id_{\theta_a}\ot(\id_{f(a)'(\theta_b)}\ot\tau^{f(a),f(b)}_{H(V)})(s_V\ot\id_{U'_{f(a),f(b)}}))\circ\\
    & (\beta^a_{b_*(V)}\ot\id_{f(a)'(\theta_b)\ot U'_{f(a),f(b)}}),\\
    s_V&=\zeta^{f(a)}_{\theta_b,f(b)'(H(V))}\circ f(a)'(\beta^b)\circ (\zeta^{f(a)}_{H(b_*(V)),\theta_b})^{-1},\\
    q_V&=(\id_{H(U_{a,b})}\ot\beta^{ab}_V) (\xi_{U_{a,b},(ab)_*(V)}\circ H(\sigma^{ab}_V)\circ ({\xi}_{a_*b_*(V),U_{ab}})^{-1}\ot\id_{\theta_{ab}}).\end{align*}
\qed
\end{itemize}
\end{teo}

Given the collection $((H,\xi),f,(\theta_a,\beta^a), \chi_{a,b})_{a,b\in\Gamma}$ as in the previous Theorem, 
the monoidal equivalence $F:\ca(\Sigma)\to \ca(\Sigma')$ is defined by
$$F([V,a])=[H(V)\ot\theta_a,f(a)],\quad [V,a]\in \ca(\Sigma),$$
for any $[V,a] \in \ca(\Sigma)$.

\begin{rmk} 
  In \cite{Ga1} the author defines crossed systems in terms of equivalence
classes of monoidal functors, up to monoidal isomorphisms, and equivalence classes
of pseudo-natural isomorphisms, up to invertible modifications. This is done this way
since it is shown that equivalence classes of crossed product extensions of the tensor category
$\ca$ by the group $\Gamma$ are classified by crossed systems. Since we are only interested in giving examples,
our definition of crossed systems is a \emph{representative} of a crossed systems according to \cite{Ga1}.
\end{rmk}

\subsection{Coherent outer actions for the corepresentation category of a
 Hopf algebra}
Let $H$ be a finite-dimensional Hopf algebra.
We shall give an explicit description for coherent outer actions on the
tensor category $\Comod(H)$ of finite-dimensional left $H$-comodules in terms of
Hopf algebraic data.
Let $\Gamma$ be a finite group. \medbreak

Let us fix the following notation. If $g\in G(H)$ and $L$ is a
$(H,H)$-biGalois object then the cotensor product $L\Box_H \ku_g$
is one-dimensional. Let $\phi(L,g)\in \Gamma$ be the group-like element such that
$L\Box_H \ku_g\simeq \ku_{\phi(L,g)}$ as left $H$-comodules.

\begin{lema}\label{lm:cros-sys-comod} Assume that 
 $\Upsilon=(L_a,(g(a,b),f^{a,b}),\gamma_{a,b,c})_{a,b,c\in \Gamma}$ is a collection where
\begin{itemize}
\item   for any $a\in \Gamma$, $L_a$ is a $(H,H)$-biGalois object;
\item  $g(a,b)\in G(H)$  is a group-like element and  $f^{a,b}:(L_a\Box_H L_{b})^{g(a,b)}\to L_{ab}$
are bicomodule algebra isomorphisms;
\item $\gamma_{a,b,c}\in\ku^{\times}$,
\end{itemize}
such that:
\begin{equation}\label{cross-hopf-11}
L_1=H, \quad (g(1,a),f^{1,a})=(1,\id_{L_a})= (g(a,1),f^{a,1});
\end{equation}
\begin{equation}\label{cross-hopf-12}\phi(L_a,g(b,c)) g(a,bc)= g(a,b) g(ab,c);
\end{equation}
\begin{equation}\label{cross-hopf-13}
\gamma_{a,1,b}=\gamma_{1,a,b}=\gamma_{a,b,1}=1;
\end{equation}
\begin{equation}\label{cross-hopf-14}
(f^{a,b}\ot\id_{L_c})f^{ab,c}=(\id_{L_a}\ot f^{b,c})f^{a,bc},
\end{equation}
 for all $a,b,c\in \Gamma$. Associated to such  $\Upsilon$ there is a crossed system $\overline{\Upsilon}$
of $\Gamma$ over $\Comod(H)$.
Moreover, the crossed system $\overline{\Upsilon}$ is a coherent outer action
  on $\Comod(H)$ if and only if $\gamma$ is a 3-cocycle, that is, for all $a,b,c,d\in\Gamma$
\begin{equation}\label{3-cocy-gamma}
\gamma_{a,b,c}\gamma_{a,bc,d}\gamma_{b,c,d}=\gamma_{ab,c,d}\gamma_{a,b,cd}.
\end{equation}
\end{lema}
\begin{proof} For any $a, b\in \Gamma$ define the monoidal functor
$a_*:\Comod(H)\to \Comod(H)$, $a_*=L_a\Box_H -$ and $U_{a,b}=\ku_{g(a,b)}$.

 \medbreak

Define the pseudo-natural isomorphism
$(\ku_{g(a,b)},\sigma^{a,b}):a_*\circ b_* \to (ab)_*$
which comes from the  bicomodule algebra isomorphism
$$f^{a,b}:(L_a\Box_H L_{b})^{g(a,b)}\to L_{ab}$$ as explained in
Remark \ref{pseudo=iso-comod}. \medbreak

The existence of the  isomorphisms
$\gamma_{a,b,c}: L_a \Box_H \ku_{g(b,c)} \to \ku_{g(a,b) g(ab,c)}$
is equivalent to $\phi(L_a,g(b,c)) g(a,bc)= g(a,b) g(ab,c)$. Since both vector spaces
 $L_a\Box_H \ku_{g(b,c)}\ot\ku_{g(a,bc)}$ and $\ku_{g(a,b)}\ot\ku_{g(ab,c)}$
are one-dimensional, the map  $\gamma_{a,b,c}:\ku\to\ku$
is given by multiplication of a scalar $\gamma_{a,b,c}\in\ku^{\times}$.

Equation \eqref{crossed-syst11} is equivalent to
 \eqref{cross-hopf-11},  \eqref{crossed-syst12} is equivalent to  \eqref{cross-hopf-13},
and Equation \eqref{crossed-syst14} is equivalent to \eqref{cross-hopf-14}.
Since $f^{a,b}$ is an algebra morphism then Equation \eqref{crossed-syst13} is satisfied.
Equation
\eqref{3-cocy-gamma} follows from \eqref{eq-acts}.
\end{proof}

\begin{defi} Given a collection $\Upsilon$ as in the previous Lemma,
define $\Comod(H)(\Upsilon):=\Comod(H)(\overline{\Upsilon})$ the
 $\Gamma$-crossed product tensor category
associated to the coherent outer action$\overline{\Upsilon}$.
\end{defi}
The next Lemma is a direct consequence of the Theorem
\ref{th:eq-mon} applied to $\ca=\Comod(H)$.

Assume  that
$\Upsilon=(L_a,(g(a,b),f^{a,b}),\gamma_{a,b,c})_{a,b,c\in\Gamma}$ and

$\Upsilon'=(L_a',(g'(a,b),z^{a,b}),\gamma'_{a,b,c})_{a,b,c\in\Gamma}$
are collections satisfying conditions given in  Lemma \ref{lm:cros-sys-comod}. 
Thus, the associated objects $\overline\Upsilon,\overline{\Upsilon'}$ are coherent outer $\Gamma$-actions).

\begin{lema}\label{lm:S-S'-comod}

 Any monoidal equivalence $F:Comod(H)(\Upsilon)\to Comod(H)(\Upsilon')$
 comes from a collection    $(L,\lambda,(h(a),h^a),\tau_{a,b})_{a,b\in\Gamma}$ where
\begin{itemize}
\item $L$ is a $(H,H)$-biGalois object,
\item $\lambda:\Gamma\to\Gamma$ is a group isomorphism,
\item $h(a)\in G(H)$ is a group-like and $h^a:(L\Box_H L_a)^{h(a)}\to L_{\lambda(a)}'\Box_H L$ is a 
biGalois object isomorphism, satisfying $(h(1),h^1)=(1,\id)$,
\item $\tau_{a,b}\in\ku^{\times}$ satisfies $\tau_{a,1}=\tau_{1,a}=1$,
\end{itemize}
and also the following equations are fulfilled
\begin{equation}\label{eq:tau0}
\phi(L,g(a,b))h(ab)=h(a)\phi(L_{\lambda(a)}',h(b))g'(\lambda(a),\lambda(b)),\end{equation}
\begin{equation}\label{eq:tau}
h^{ab}(\id_L\ot f^{a,b})=(z^{\lambda(a),\lambda(b)}\ot\id_L)(\id_{L_{\lambda(a)}'}\ot h^b)(h^a\ot\id_{L_b}).\end{equation}
\end{lema}\qed

\section{Examples of $C_2$-extensions of $\Comod(\Ac(V,u,C_2))$}\label{section:crossed-spg}

Let $C_2$ be the cyclic group of 2 elements. 
In this section we shall give  explicit examples of tensor categories that are $C_2$-extensions of
the tensor category $\Comod(\Ac(V,u,C_2))$ with $V$ a 2-dimensional vector space.

\subsection{$C_2$-extensions of $\Comod(H)$}\label{pcs}

Let $H$ be a finite-dimensional Hopf algebra.
First, we explicitly describe data giving rise to $C_2$-extensions of $\Comod(H)$
in the particular case the group of group-like elements of the Hopf algebra $H$ is a cyclic group
of order 2 generated by $u$.

\medbreak

Assume that $(L,g,f,\gamma)$ is  a collection where
\begin{itemize}
 \item $L$ is a $(H,H)$-biGalois object;

\item $g\in G(H)$ is a group-like element such that $\varpi:L\Box_H \ku_g\simeq \ku_g$ as left $H$-comodules;

\item $f:(L\Box_H L)^g\to H$ is a bicomodule algebra isomorphism and
\item  $\gamma\in\ku^{\times}$, $\gamma^2=1.$
\end{itemize}
According to Lemma \ref{lm:cros-sys-comod} from data $(L,g,f,\gamma)$
we obtain a crossed system of $C_2$ over $\Comod(H)$.
Just take $L_u=L$, $L_1=H$, $g(u,u)=g$, $1=g(1,u)=g(u,1)=g(1,1)$, $f^{u,u}=f$, $f^{1,u}=f^{u,1}=f^{1,1}=\id$ and $\gamma_{a,b,c}=1\in\ku$ for any $a,b,c\in C_2$
  except $\gamma_{u,u,u}:=\gamma$. Let us denote this crossed system $\Upsilon$.

The monoidal structure of the category $\Comod(H)(\Upsilon)$, given by Theorem \ref{ten-prod}
 explicitly reads as follows. For any $V,W,Z\in\Comod(H)$ and $b\in C_2$:
\begin{align*}
[V,1]\ot[W,b]&=[V\otk W,b],\\
[V,u]\ot[W,1]&=[V\otk (L\Box_H W),u],\\
[V,u]\ot[W,u]&=[V\otk (L\Box_H W)\otk \ku_g,1],
\end{align*}
The unit object is $[\ku, 1]$ and dual objects are given by $$([V,1])^*=[V^*,1],\quad ([\ku,1])^*=[\ku,1]\quad \text{ and }\quad ([\ku,u])^*=[\ku_{g^{-1}},u].$$
Finally, the associativity, on elements of the form $[V,u]$, is given by
$$\alpha_{[V,u][W,u][Z,u]}=\gamma(\id_{V\ot L\Box_H W}\ot f\Box_H \id_Z\ot\id_{\ku_g})(\id_{V\ot L\Box_H W\ot L\Box_H L \Box_H Z}\ot \varpi)\circ$$
$$(\id_{V\ot L\Box_H W}\ot \xi_{L\Box_H Z,\ku_g})(\id_V\ot\xi_{W,L\Box_H Z\ot \ku_g}).$$
The other components of the associativity are trivials. \eq Here $\xi=(\xi^{L})^{-1}$ is the morphism defined in the Equation \eqref{iso-gal}.\eq

\subsection{Explicit examples of $C_2$-extensions of $\Comod(\Ac(V,u,C_2))$}
In this section $H=\Ac(V,u,C_2)$ where $V$ is a 2-dimensional vector space.
Using the results of previous sections, we describe families of crossed
 systems of $C_2$ over $\Comod(\Ac(V,u,C_2))$. These crossed systems come from a collection
 $(L,g,f,\gamma)$ as presented in Section \ref{pcs}.
Below, we present two such families
depending on the biGalois object $L$. For the first family the biGalois object $L$ is
the one presented  in Section \ref{sec:exam-bigal} and for the second family the biGalois object $L$
is trivial.

\begin{lema}\label{l-tec} Let be $\xi,\gamma\in\ku,g\in C_2,$ and let $f\in Hom(H^g,H)$ be a comodule algebra
 isomorphism. Assume $\gamma^2=1$.
\medbreak
\begin{itemize}
 \item[1.] The collection $(\xi,g,f,\gamma)$ has associated a coherent outer
 $C_2$-action over $\Comod(\Ac(V,u,C_2))$ and the corresponding $C_2$-crossed
 product tensor category will be denoted by $\ca_\xi(g,f,\gamma)$.

 \item[2.] The collection $(g,f,\gamma)$ has associated a coherent outer
 $C_2$-action over $\Comod(\Ac(V,u,C_2))$ and the corresponding  $C_2$-crossed
product tensor category will be denoted by $\Do(g,f,\gamma)$.
\end{itemize}

\end{lema}
\begin{proof} 1. Let $L=\textbf{U}_{\xi}$ be the $(H,H)$-biGalois object defined in section
 \ref{sec:exam-bigal}. It follows from Lemma \ref{lm:iso-esp} that
 $\textbf{U}_{\xi}\Box_H\ku_{g}\simeq \ku_{g}$.

2. Following the same idea, take $L=H$. Then $H\Box_H\ku_{g}\simeq \ku_{g}$.
\end{proof}

We want to be more explicit in the determination of the comodule algebra isomorphism
$f:H^g\to H$ that appears in Lemma \ref{l-tec}. We make use of the proof of Proposition \ref{bi-inn} (1), where 
such comodule 
algebra maps are explicitly determined.
Let $(\xi,g,f,\gamma)$ be a collection as in Lemma \ref{l-tec}.
There are two options:
\begin{itemize}
\item If $g=1$, then $f:H\to H$. Let $\delta:H\to\ele(\Id ,0,\id ,1)$ be the canonical isomorphism
$h\mapsto (h,h)$ and define $\overline f:=\delta\circ f\circ\delta^{-1}$. By (the proof of) Proposition \ref{bi-inn}
 $$\overline f:\ele(\Id_V ,0,\id,1)\to \ele(\Id_V ,0,\id,1),$$

satisfies that $\overline f(x,y)=(x,y)$ if $(x,y)\in \{(v,v):v\in V\}$ which implies that $f(x)=x$
if $x\in V$. Moreover $\overline f(e_{1,1})=\chi_1e_{1,1}=e_{1,1}$ and $\overline f(e_{u,u})=\chi_ue_{u,u}=e_{u,u}$.
Then $f=\id_H$.
\item If $g=u$, then $f:H^u\to H$. By (the proof of) Proposition \ref{bi-inn} (1)
 $$\overline{f}:\ele(\Id_V^u,0,\id,1)\to \ele(\Id_V,0,\id,1),$$
satisfies that $\overline{f}(x,y)=(x,-y)$ if $(x,y)\in \{(u\cdot v,v)|v\in V\}$
 which implies that $f(x)=u\cdot x=-x$ if $x\in V$. Moreover
$\overline{f} (e_{1,1})=e_{1,1}$ and $\overline{f}(e_{u,u})=\chi_ue_{u,u}=-e_{u,u}$, so $f(u)=-u$.
We shall denote by $\iota:H^u\to H$ this unique bicomodule algebra isomorphism.
\end{itemize}
\medbreak

Hence, we obtain four
 families of $C_2$-crossed product tensor categories
\begin{equation}\label{fam-cross}
\ca_\xi(1,\id,\gamma),\ca_\xi(u,\iota,\gamma),\Do(1,\id,\gamma),\Do(u,\iota,\gamma).
\end{equation}

Some of these tensor categories are equivalent. We shall use
  Lemma \ref{lm:S-S'-comod} to distinguish them.

\begin{teo} Let be $\xi,\xi',\gamma,\gamma'\in\ku$ with $\gamma^2=1=(\gamma')^2$. As tensor categories
\begin{align*}
\ca_\xi(1,\id,\gamma)&\ncong \ca_{\xi'}(u,\iota,\gamma'),\quad  \ca_\xi(1,\id,\gamma)\cong \ca_0(1,\id,\gamma'), \\ 
\ca_\xi(u,\iota,\gamma)&\cong \ca_0(u,\iota,\gamma'), \quad\Do(1,\id,\gamma)\ncong \Do(u,\iota,\gamma'),
\\
 \Do(1,\id,\gamma)&\ncong \ca_0(1,\id,\gamma'), \quad\Do(u,\iota,\gamma)\ncong \ca_0(u,\iota,\gamma').
\end{align*}\end{teo}

\begin{proof}  Using  Lemma \ref{lm:S-S'-comod}, there exists a monoidal equivalence
$$\ca_\xi(g,f,\gamma)\simeq \ca_{\xi'}(g',f',\gamma')$$ if there exists
\begin{enumerate}
\item $L=\ele(T,0,\alpha,1)$ a biGalois object over $H$,
\item $h:=h(u)\in C_2$ and $h^u:\ele(T_hTT_\xi,0,\id,1)\to\ele(T_{\xi'}T,0,\id,1)$ a biGalois  isomorphism,  
\item $\tau:=\tau_{u,u}\in\ku^{\times}$,
\end{enumerate}
 satisfying
\begin{equation}\label{cond-iso}
\alpha(g)=g',\quad \Phi(\id_L\ot f)=(f'\ot\id_L)(\id_{\textbf{U}_{\xi'}}\ot h^u)(h^u\ot\id_{\textbf{U}_\xi}),
\end{equation}
where $\Phi:L\Box_H H\to H\Box_H L$ is the isomorphism given by $l\ot h\mapsto l_{-1}\ot l_0\varepsilon(h)$.

The second condition of \eqref{cond-iso} comes from Equation \eqref{eq:tau},
 and the first condition from Equation \eqref{eq:tau0}:

For all $a,b\in C_2$, $L\Box_h\ku_{g(a,b)}\simeq\ku_{\alpha g(a,b)}$ and $L_a'\Box_H\ku_{h(b)}\simeq \ku_{h(b)}$, then Equation \eqref{eq:tau0} implies that $\alpha(g(a,b))h(ab)=h(a)h(b)g'(a,b)$. For $a=1$ or $b=1$ this equation is valid. For $a=u=b$, we obtain $\alpha(g)=h^2g'=g'$.

Since $\alpha=\id$, we obtain that $\ca_\xi(1,\id,\gamma)\ncong \ca_{\xi'}(u,\iota,\gamma')$.
\ \\

By Lemma \ref{bi-inn}(1), $h^u$ is an isomorphism if and only if 
$$T_hTT_\xi=T_{\xi'}T, \quad \text{or } \quad T_uT_hTT_\xi=T_{\xi'}T.$$

To prove that there is a monoidal equivalence $\ca_\xi(1,\id,\gamma)\simeq \ca_0(1,\id,\gamma')$ choose $h=1$ and $h^u=\id$ then $TT_\xi=T_{0}T$ if $T$ is given by the matrix
$$\begin{pmatrix}                                                 1 & \xi/2 \\
                                                                  0 & 1 \\
                                                                \end{pmatrix}.$$
We only need to check that $$\Phi(\id_L\ot \varphi_2)=(\varphi_3\ot\id_L)(\id_{U_{0}}\ot\varphi_1)(\varphi_1\ot\id_{\textbf{U}_\xi}),$$
where
\begin{itemize}
\item $\varphi_1:L\Box_H \textbf{U}_\xi\to \textbf{U}_{0}\Box_H L$, coming from $h^u=\id:\ele(TT_\xi,0,\id,1)\to\ele(T_0T,0,\id,1)$ up to isomorphism,
\item $\varphi_2:\textbf{U}_\xi\Box_H \textbf{U}_\xi\to H$, coming from $\id:\ele(T^2_\xi ,0,\id,1)\to\ele(\id,0,\id,1)$, which satisfies $(\varphi_2)^{-1}(v)=(T_\xi T_\xi(v),T_\xi(v))\ot 1+e_{u,u}\ot(T_\xi(v),v)$ and $(\varphi_2)^{-1}(e_{g,g})=e_{g,g}\ot e_{g,g}$ for $v\in V$ and $g\in C_2$,
\item $\varphi_3:\textbf{U}_{0}\Box_H \textbf{U}_{0}\to H$ coming from $\id:\ele(T^2_0,0,\id,1)\to\ele(\id,0,\id,1)$ up to isomorphism.\end{itemize}

Let $v\in V$. If $a=(TT_\xi(v),T_\xi(v))\ot 1+e_{u,u}\ot (T_\xi(v),v)\in L\Box_H \textbf{U}_\xi$ then $\varphi_1(a)=(T_0T(v),T(v))\ot 1+e_{u,u}\ot(T(v),v)$:

Let $\zeta_1:\ele(TT_\xi,0,\id,1)\to L\Box_H \textbf{U}_\xi$ and $\zeta_2:\ele(T_0T,0,\id,1)\to \textbf{U}_{0}\Box_H L$ be the isomorphisms give in the Lemma \ref{bi-inn}(3), whose satisfy
$$\zeta_1(TT_\xi(v),v)=TT_\xi(v),T_\xi(v))\ot 1+e_{u,u}\ot (T_\xi(v),v),$$
$$\zeta_2(T_0T(v),v)=(T_0T(v),T(v))\ot 1+e_{u,u}\ot(T(v),v).$$
By definition of $\varphi_1$, we have that $\varphi_1\circ \zeta_1=\zeta_2\circ \id_{\ele(TT_\xi,0,\id,1)}$, and this implies the claim.

\medbreak

By the same argument, if $b=(T_0T_0(v),T_0(v))\ot 1+e_{u,u}\ot(T_0(v),v)\in \textbf{U}_{0}\Box_H \textbf{U}_{0}$ then $\varphi_3(b)=v$.

\medbreak

Moreover $\Phi=\alpha_1\circ\alpha_2$ where $\alpha_1:L\to H\Box_H L$, $\alpha_2:L\Box_H L\to L$ and $(\alpha_1)^{-1}(h\ot l)=\varepsilon(h)l$ and $(\alpha_2)^{-1}(l)=l_0\ot l_1$.
\ \\

Let $x=(T(w),w)\in L$, then
$$(\alpha_1)^{-1}(\varphi_3\ot\id_L)(\id_{\textbf{U}_{0}}\ot\varphi_1)(\varphi_1\ot\id_{\textbf{U}_\xi})(\id_L\ot(\varphi_2)^{-1})(\alpha_2)^{-1}(x)=x,$$
since \begin{align*}
x&\mapsto e_{u,u}\ot w+(T(w),w)\ot 1\\
&\mapsto e_{u,u}\ot e_{u,u}\ot (T_\xi(w),w)+e_{u,u}\ot(T_\xi T_\xi(w),T_k(w))\ot 1+ (T(w),w)\ot 1\ot 1\\
&\mapsto e_{u,u}\ot e_{u,u}\ot (T_\xi(w),w)+(T_0TT_\xi(w),TT_\xi(w))\ot 1\ot 1\\&+e_{u,u}\ot(TT_\xi(w),T_\xi(w))\ot 1\\
&\mapsto (T_0TT_\xi(w),TT_\xi(w))\ot 1\ot 1+e_{u,u}\ot(T_0T(w),T(w))\ot 1\\&+e_{u,u}\ot e_{u,u}\ot (T(w),w)\\
&\mapsto u\ot(T(w),w)+T(w)\ot 1\\
&\mapsto x.
\end{align*}
In the same way, $(g,g)\mapsto (g,g)$ for all $g\in C_2$. Which implies that $\ca_\xi(1,\id,\gamma)\simeq \ca_0(1,\id,\gamma')$.

To prove $\ca_\xi(u,\iota,\gamma)\simeq \ca_0(u,\iota,\gamma')$, is enough to take $h=u$ and $h^u:\ele(T_uTT_k,0,\id,1)\to \ele(T_0T,0\id,1)$ given for $x,y\in V$ by
$$h^u(x,y)=(x,-y), \quad h^u(e_{u,u})=-e_{u,u}.$$

 It follows from  Lemma \ref{lm:S-S'-comod}, that there is a monoidal equivalence
 $$\Do(1,\id,\gamma)\simeq \Do(u,\iota,\gamma')$$  if 
and only if
there exist $M=\ele(R,0,\alpha,1)$ a biGalois object over $H$,
$h\in C_2$, $h^u:\ele(T_hR,0,\id,1)\to\ele(R,0,\id,1)$ a biGalois object isomorphism and $\tau\in\ku^{\times}$. 
As before, they have to satisfy that $\alpha(1)=u$, but $\alpha=\id$. 
This proves that $\Do(1,\id,\gamma)\ncong \Do(u,\iota,\gamma')$.

Again, using Lemma \ref{lm:S-S'-comod}, $\Do(1,\id,\gamma)\simeq \ca_0(1,\id,\gamma')$ as monoidal categories if 
and only if there exist $M=\ele(R,0,\alpha,1)$ a biGalois object over $H$, $h\in C_2$, $h^u:\ele(T_hR,0,\id,1)\to\ele(T_{0}R,0,\id,1)$ a biGalois object isomorphism and $\tau\in\ku^{\times}$.

By Lemma \ref{bi-inn}(3), $h^u$ is an isomorphism if 
and only if $T_hR=T_{0}R$ or $T_uT_hR=T_{0}R$, but the last two equations do not have a solution for $T$ invertible. So $\Do(1,\id,\gamma)\ncong \ca_0(1,\id,\gamma')$ and $\Do(u,\iota,\gamma)\ncong \ca_0(u,\iota,\gamma')$.
\end{proof}

In conclusion, we obtain eight pairwise non-equivalent tensor categories
\begin{equation}\label{ex-tc}
\begin{split}
\ca_0(1,\id,1),\ca_0(1,\id,-1),\ca_0(u,\iota,1),\ca_0(u,\iota,-1), \\
\Do(1,\id,1),\Do(1,\id,-1),\Do(u,\iota,1),\Do(u,\iota,-1).
\end{split}
\end{equation}

\subsection{Explicit description of the monoidal structure}\label{cat-exa}
Using Theorem \ref{ten-prod}, we can explicitly describe the  tensor product and the
associativity constraint for the eight tensor categories presented above.
Recall that all those categories
have the same underlying Abelian category $\Comod(\Ac(V,u,C_2)) \oplus
\Comod(\Ac(V,u,C_2))$ where $V$ is a 2-dimensional vector space. The associativity
constraints that we describe are the non-trivial ones.

\bigbreak

 Let $V,W,Z\in \Comod(\Ac(V,u,C_2))$ and $g\in C_2$.

$\bullet$ The tensor product, dual objects and associativity in the category $\ca_0(1,\id,\pm 1)$ are given by
\begin{align*}
[V,1][W,g]&=[V\ot W,g] ,& [V,u][W,g]&=[V\ot \textbf{U}_0\Box_H W,ug],\\
[V,1]^*&=[V^*,1],& [\mathbf{1},u]^*&=[\ku,u],\end{align*}
$\alpha_{[V,u],[W,u],[Z,u]}=[\pm(\id_{V\ot\textbf{U}_0\Box W} \ot\epsilon\varphi_2\ot
\id_{Z})(\id_V\ot \xi_{W,\textbf{U}_0\Box Z}),u].$
\smallbreak

 Here $\xi=(\xi^{\textbf{U}_0})^{-1}$ is the morphism defined in the Equation \eqref{iso-gal}.

$\bullet$ The tensor product, dual objects and associativity in $\ca_0(u,\iota,\pm 1)$ are given by
\begin{align*}
[V,1][W,1]&=[V\ot W,1] ,& [V,u][W,u]&=[V\ot \textbf{U}_0\Box_H W\ot\ku_u,1],\\
[V,1][W,u]&=[V\ot W,u] ,& [V,u][W,1]&=[V\ot \textbf{U}_0\Box_H W,u],\\
[V,1]^*&=[V^*,1],& [\mathbf{1},u]^*&=[\ku_u,u],\end{align*}
The associativity constraint $\alpha_{[V,u],[W,u],[Z,u]}$ is equal to 
$$[\pm(\id_{V\ot\textbf{U}_0\Box W} \ot
(\epsilon\iota\varphi_2\ot\id_{Z\ot\textbf U_0\Box \ku_u})(\xi_{\textbf{U}_0\Box Z,\ku_u}))
(\id_V\ot \xi_{W,\textbf{U}_0\Box Z\ot\ku_u}),u].$$

\smallbreak
$\bullet$ The tensor product, dual objects and associativity in $\Do(1,\id,\pm 1)$ are given by
\begin{align*}
[V,1][W,g]&=[V\ot W,g] ,& [V,u][W,g]&=[V\ot W,ug],\\
[V,1]^*&=[V^*,1],& [\mathbf{1},u]^*&=[\ku,u],\end{align*}
$\alpha_{[V,u],[W,u],[Z,u]}=[\pm(\id_{V\ot W\ot Z},u].$\smallbreak

$\bullet$ The tensor product, dual objects and associativity in $\Do(u,\iota,\pm 1)$ are given by
\begin{align*}
[V,1][W,1]&=[V\ot W,1] ,& [V,u][W,u]&=[V\ot W\ot\ku_u,1],\\
[V,1][W,u]&=[V\ot W,u] ,& [V,u][W,1]&=[V\ot W,u],\\
[V,1]^*&=[V^*,1],& [\mathbf{1},u]^*&=[\ku_u,u],\end{align*}
$\alpha_{[V,u],[W,u],[Z,u]}=[\pm(\id_{V\ot W} \ot\varepsilon\iota\ot\id_{Z\ot \ku_u},u].$

\subsection{Frobenius-Perron dimension of the $C_2$-crossed extensions. }
For a review on Frobenius-Perron dimension we refer to \cite{EGNO}. For any object $X$ in a category 
$\ca$ we denote by $\langle X \rangle$ the class of $X$ in the Grothendieck group of $\ca$.

For the categories presented in \eqref{ex-tc}, the isomorphism classes of the simple objects are 
$$\langle [\ku_1,1]\rangle,\quad \langle [\ku_1,u]\rangle,\quad \langle [\ku_u,1]\rangle,\quad \langle [\ku_u,u]\rangle.$$
Using Theorem \ref{th:proy-super}, the projective covers of these simple objects are respectively
$$\langle [P_1,1]\rangle,\quad \langle [P_1,u]\rangle,\quad \langle [P_u,1]\rangle,\quad \langle [P_u,u]\rangle.$$
Using Corollary \ref{grot} it follows from  a straightforward computation that in any of the categories listed in \eqref{ex-tc}
 $$\fpd \langle [ \ku_g,h]\rangle=1, \quad \fpd\langle[ P_g,h]\rangle=4,$$
 for any $g,h\in C_2$. This implies the next result.
\begin{teo}\label{pd} If $\ca$ is any of the tensor categories listed in \eqref{ex-tc} then 
$\fpd \ca=16.$\qed
\end{teo}

The above Theorem implies, using  \cite[Proposition 1.48.2]{EGNO}, that all the tensor categories 
listed in \eqref{ex-tc} are representation categories of quasi-Hopf algebras.

\subsection*{Acknowledgments}  This work  was  partially supported by
 CONICET, Secyt (UNC), Mincyt (C\'ordoba) Argentina. We are grateful to C\'esar Galindo for
patiently answering all our questions on his work \cite{Ga1}.
M.M thanks M. Suarez-Alvarez for some comments on an earlier
version of the paper. We also thank the referee for the careful reading of the paper and for 
his many constructive observations.

\end{document}